\documentclass[a4,11pt]{article}

\usepackage{amssymb,amsfonts,amsmath,latexsym,color, amscd} 
\usepackage{url}

\usepackage[dvips]{psfrag,graphicx} 

\usepackage{cases}

\newtheorem{theorem}{Theorem}[section]
\newtheorem{lemma}[theorem]{Lemma}
\newtheorem{corollary}[theorem]{Corollary}
\newtheorem{proposition}[theorem]{Proposition}

\newenvironment{proof}{{\par\addvspace{0.1cm}\noindent \bf Proof. }}{\hfill$\Box$\par\medskip}

\newtheorem{remark}[theorem]{Remark}

\numberwithin{equation}{section}

\def\e{\varepsilon}
\def\R{\Re\mathfrak{e} \,}
\def\RR{\mathbb{R}}
\def\CC{\mathbb{C}}
\def\QQ{\mathbb{Q}}

\def\Res{\mbox{\rm Res}}

\def\j#1{\textcolor[named]{Plum}{\bf {#1}}} 

\begin{document}

\title{Characterization of balls by generalized Riesz energy}

\author{Jun O'Hara\footnote{Supported by JSPS KAKENHI Grant Number 16K05136.}}
%
%
\maketitle

\begin{abstract}
We show that balls, circles and $2$-spheres can be identified by generalized Riesz energy among compact submanifolds of the Euclidean space that are either closed or with codimension $0$
, where the Riesz energy is defined as the double integral of some power of the distance between pairs of points. 
As a consequence, we obtain the identification by the interpoint distance distribution. 
\end{abstract}



\section{Introduction}
%
Suppose $X$ is a compact submanifold of $\RR^d$ which is either a {\em compact body} $\Omega$, i.e. the closure of a bounded open set of $\RR^d$, or a closed submanifold $M$. 
Let us consider the integral 
\begin{equation}\label{I_q}
I_q(X)=\int_{X\times X}|x-y|^q\,dxdy,
\end{equation}
where $dx$ and $dy$ are the Lebesgue measures of $X$. 
It is well-defined if $q>-\dim X$. 
It is called the {\em Riesz $q$-energy} of $X$ when $X$ is a compact body and $-d<q<0$. 

Fix a submanifold $X$ and consider the power $q$ in the integral as a complex number, denoted by $z$ in what follows. 
Then \eqref{I_q} is well-defined on a domain $\{z\in\CC\,:\,\R z>-\dim X\}$, where the map $z\mapsto I_z(X)$ is holomorphic. 
Extend the domain of \eqref{I_q} by analytic continuation to a region of $\CC$, which depends on the regularity of $X$ (it is the whole complex plane $\CC$ if $X$ is smooth). 
Then we obtain a meromorphic function with only simple poles at some negative integers. 
We denote it by $B_X(z)$ and call it {\em Brylinski's beta function} of $X$, as it can be expressed by the beta function when $X$ is a circle, sphere or a ball. 
It was introduced by Brylinski \cite{B} for knots, studied by Fuller and Vemuri \cite{FV} for closed (hyper-)surfaces, and by Solanes and the author \cite{OS2} for compact bodies. 

The beta function provides geometric quantities of $X$. For example, the volumes of $X$ and of the boundary $\partial X$ if exists, the total squared curvature of closed curves or the Willmore functional of closed surfaces as residues, and some kind energies as values at special $z$'s. 
With these quantities, we are inclined to ask a question to what extent a space $X$ can be identified by the beta function $B_X(z)$. 
We begin with introducing some preceding results on the identification by closely related geometric quantities. 

Let $f_X(r)$ be the {\em interpoint distance distribution} of $X$;
\[f_X(r)=\mbox{Vol}\left(\{(x,y)\in X\times X\,:\,|x-y|\le r\}\right).\]
%
%
It is equivalent to the integral \eqref{I_q} in the sense that the Mellin transform of $f_X^{\,\prime}$ is equal to $I_{q-1}(X)$; 
\begin{equation}\label{Mf}
(\mathcal{M}f_X^{\,\prime})(q)=\int_0^\infty r^{q-1}f_X^{\,\prime}(r)\,dr=I_{q-1}(X), 
\nonumber
\end{equation}
and hence 
\[
f_X^{\,\prime}(r)=\left(\mathcal{M}^{-1}I_{\bullet-1}(X)\right)(r)
=\frac1{2\pi i}\int_{c-i\infty}^{c+i\infty}r^{-z}I_{z-1}(X)\,dz \hspace{0.5cm}(c>1-\dim X). 
\]

\smallskip
The {\em chord length distribution} of a {\sl convex} body $K$ is given by 
\[ 
g_K(r)=\mu\{\ell\in\mathcal{E}_1\,:\,L(K\cap\ell)\le r\}, 
\]
where $\mathcal{E}_1$ is the set of lines in $\RR^d$, $\mu$ is a measure on $\mathcal{E}_1$ that is invariant under motions of $\RR^d$, and $L$ means the length. 
It is equivalent to the interpoint distance distribution for convex bodies in the sense that $g_K$ uniquely determines and is uniquely determined by $f_K$ (for example, \cite{M} p.25), which is a consequence of the Blaschke-Petkantschin formula (for example, \cite{San2} (4.2) p.46). 

\medskip
Let us first consider the identification problem of $X$ by the interpoint distance distribution; whether $f_X(r)=f_{X'}(r)$ for any $r$ implies $X=X'$ up to motions of $\RR^d$. 
The picture is quite different according to whether we assume the convexity of $X$ or not, although the answer is negative in both cases. 

In fact, for convex bodies, Mallows and Clark \cite{MC} gave a pair of non-congruent convex planar polygons with the same chord length distribution as illustrated in Figure \ref{MC}, 
\begin{figure}[htbp]
\begin{center}
\includegraphics[width=.5\linewidth]{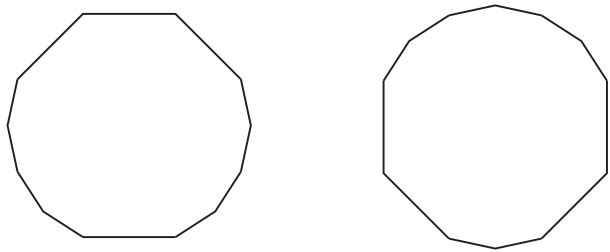}
\caption{Mallows and Clark's counter-example}
\label{MC}
\end{center}
\end{figure}
whereas Waksman \cite{W} pointed out that it is exceptional by showing that a ``{\em generic}'' planar convex polygon can be identified by the chord length distribution. 

On the other hand, for general case, Caelli \cite{C} gave a method to produce pairs of non-congruent subsets of $\RR^2$, which are not convex in general, with the same interpoint distance distribution by using two axes of symmetry, as is illustrated in Figure \ref{caelli}. 
\begin{figure}[htbp]
\begin{center}
\includegraphics[width=.35\linewidth]{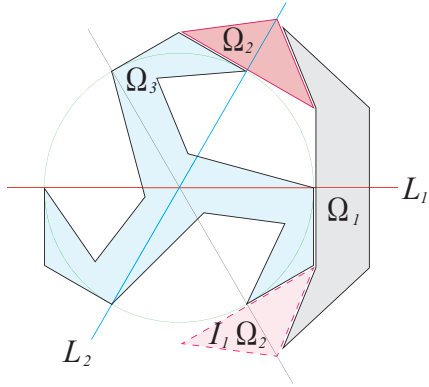}
\caption{Let $I_1$ and $I_2$ be reflections in lines $L_1$ and $L_2$ respectively, which form the angle $q\pi$ $(q\in\QQ)$. Then $R=I_1I_2$ is the rotation by angle $2q\pi$. Let $\Omega_1, \Omega_2$ and $\Omega_3$ be mutually disjoint regions satisfying $I_1\Omega_1=\Omega_1$, $I_2\Omega_2=\Omega_2$, $R\,\Omega_3=\Omega_3$ and $I_1\Omega_3\ne\Omega_3$. Then $X=\Omega_1\cup\Omega_3\cup\Omega_2$ and $X'=\Omega_1\cup\Omega_3\cup R\,\Omega_2$ are not congruent, although they have the same interpoint distance distribution since $R\,\Omega_2=I_1\Omega_2$ implies $X'=\Omega_1\cup R\,(\Omega_3\cup \Omega_2) = \Omega_3\cup I_1(\Omega_1 \cup \Omega_2)$. This is a picture after Caelli's paper. }
\label{caelli}
\end{center}
\end{figure}

\medskip
Let us next consider a weaker problem, whether balls and spheres can be identified by the interpoint distance distribution. 
Again, the picutre is different according to whether we assume convexity or not. 

Among {\sl convex} bodies $K$, balls can be identified by the interpoint distance distribution. It follows directly from the fact that only balls give the maximum of the Riesz energy $I_q(K)$ for $-d<q<0$ among all convex bodies $K$ with a given volume $V=\sqrt{I_0(K)}$. This fact was proved by \cite{D}, \cite{San} and \cite{Sch} independently. 
There is another proof. 
The volumes of a convex body $K$ and the boundary $\partial K$ can be expressed by the chord length distribution by 
\[
\int_{\mathcal{E}_1}L(K\cap \ell)\,d\mu(\ell)=\int_0^\infty r g_K'(r)\,dr \quad\mbox{ and } \quad 
\int_{\mathcal{E}_1}\chi(K\cap \ell)\,d\mu(\ell)=\int_0^\infty g_K'(r)\,dr 
\]
up to multiplication by constants, where $\chi$ is the Euler number. It is a consequence of Crofton's intersection formula (see, for example \cite{San2} 14.3 or \cite{F} 3.2.26). 
Then the isoperimetric inequality in general dimension (\cite{F}) implies the conclusion. 

\medskip
In this paper, we drop the assumption of convexity, and instead, we assume regularity of class $C^3$, namely, we restrict ouselves to the set of compact submanifolds $X$ of $\RR^d$ of class $C^3$ with $\dim X=d$ or $\partial X=\emptyset$ ($\dim X<d$), 
and show that balls and circles can be identified by the beta function, and hence, by the interpoint distance distribution. 
We also show the identification of $2$-spheres under additional assumptions that the codimension of $X$ is not greater than $1$ and that the regularity is of class $C^4$. 

%

%
\section{Prelimanaries}
We first show that the argument in \cite{OS2} goes almost parallel even if we weaken the assumption of regularity of $X$, and introduce some preceding results on the residues of the beta function from \cite{B,FV,OS2}. 

\smallskip
Let $M$ be an $m$ dimensional closed submanifold of $\RR^d$ ($d>m$) and $x\in M$. Put 
\[
\psi_{M,x}(t)=\mbox{Vol}\left(M\cap B_x^d(t)\right) \,\mbox{ and }\> 
\varphi_{M,x}(t)=\psi_{M,x}'(t),
\]
where $B_x^d(t)$ is a $d$-ball with center $x$ and radius $t$. 
Then, 
\begin{equation}\label{beta_closed}
I_z(M)=\int_{M\times M}|x-y|^z\,dxdy=\int_0^\infty t^z\left(\int_M\varphi_{M,x}(t)\,dx\right)dt
\end{equation}
for $\R z>-m$ (\cite{OS2} Proposition 3.3). 

Let $\Omega$ be a compact body in $\RR^d$ and $x\in M$. Put 
\[
\psi_{\nu,x}(t)=\int_{\partial\Omega\cap B_x^d(t)}\langle n_x,n_y\rangle\, dy
\>\mbox{ and }\> 
\varphi_{\nu,x}(t)=\psi_{\nu,x}'(t),
\]
where $n_x$ and $n_y$ are outer unit normal vectors to $\Omega$ at $x$ and $y$. 
Then, 
\begin{eqnarray}
\displaystyle I_z(\Omega)=\int_{\Omega\times\Omega}|x-y|^z\,dxdy
&=&\displaystyle \frac{-1}{(z+2)(z+d)}\int_{\partial\Omega\times \partial\Omega}|x-y|^{z+2}\langle n_x,n_y\rangle\,dxdy \label{beta_body_Stokes} \\[2mm]
&=&\displaystyle \frac{-1}{(z+2)(z+d)}\int_0^\infty t^{z+2}\left(\int_{\partial\Omega}\varphi_{\nu,x}(t)\,dx\right)dt \label{beta_body}
\end{eqnarray}
for $\R z>-d$ and $z\neq -2$ (\cite{OS2} Lemma 4.1). 

\begin{proposition}\label{regularity} \begin{enumerate}
\item 
If $M$ is an $m$ dimensional closed submanifild of class $C^{k+1}$ $(k\ge1)$, then 
\begin{equation}\label{varphi_bar_closed}
\varphi_{M,x}(t)=t^{m-1} \, \overline{\varphi}_{M,x}(t) \nonumber
\end{equation}
for some $\overline{\varphi}_{M,x}$ of class $C^{k}$. 
Moreover, $\overline{\varphi}_{M,x}(t)$ satisfies 
\begin{equation}
\overline{\varphi}_{M,x}(0)=\sigma_{m-1}, \hspace{0.5cm}
\frac{\partial^{2i-1}\, \overline{\varphi}_{M,x}}{\partial t^{2i-1}}(0)=0 \>\>\>(1\le 2i-1 \le k). \nonumber
\end{equation}
%

\item 
If $\Omega$ is a compact body of class $C^{k+1}$ $(k\ge1)$, then 
$
\varphi_{\nu,x}(t)=t^{d-2}\,\overline\varphi_{\nu,x}(t)
$
for some $\overline\varphi_{\nu,x}$ of class $C^{k}$, which satisfies 
$\overline\varphi_{\nu,x}(0)=\sigma_{d-2}$ and $\overline\varphi_{\nu,x}^{(2i-1)}(0)=0$ $(1\le 2i-1 \le k)$.
%
\end{enumerate}
\end{proposition}

It is a $C^k$ analogue of Proposition 3.1 and Corollary 3.2 of \cite{OS2}. 

\begin{proof}
(1) Using the decomposition 
\[
\RR^d\cong T_x\RR^d=T_xM\oplus(T_xM)^\perp\cong\RR^m\oplus\RR^{d-m}
\]
we can express a neighbourhood $N_x(M)$ of $x$ of $M$ as a graph of a function from $\RR^d$ to $\RR^{d-m}$. 
Let $S^{m-1}$ be the unit sphere in $T_xM\cong\RR^m$. 
For a unit vector $v\in S^{m-1}$, let $s$ be the arc-length parameter of a curve $\gamma_{x,v}=N_x(M)\cap\left(\mbox{Span}\langle v\rangle\oplus\RR^{d-m}\right)$ with $s=0$ at point $x$ and $\gamma_{x,v}'(0)=v$. 
Let $t$ be another parameter of the curve given by the distance from the point $x$ endowed with the same signature as $s$. 
Then $s=s(x,v,t)$ is a function of $t$ of class $C^{k+1}$ defined on an open interval containing $0$. Then for small $t_0>0$, 
\[
\psi_{M,x}(t_0)= \int_0^{t_0}\left(\int_{S^{m-1}}\frac{\partial s}{\partial t}(x,v,t)\,t^{m-1}\,dv\right)dt, \]
and therefore, 
\[
\varphi_{M,x}(t_0)= t_0^{m-1}\int_{S^{m-1}}\frac{\partial s}{\partial t}(x,v,t_0)\,dv,
\]
which implies the first statement. 

Since 
$({\partial t}/{\partial s})(0)=1$ we have $({\partial s}/{\partial t})(x,v,0)=1$, 
which implies $\overline{\varphi}_{M,x}(0)=\sigma_{m-1}$. 

Since $s(x,v,t)=s(x,-v,-t)$ we have 
\[
\begin{array}{rcl}
\overline\varphi_{M,x}(t_0)&=&\displaystyle  \int_{S^{m-1}}\frac12\left(\frac{\partial s}{\partial t}(x,v,t_0)+\frac{\partial s}{\partial t}(x,-v,t_0)\right)dv \\[4mm]
&=&\displaystyle  \int_{S^{m-1}}\frac12\left(\frac{\partial s}{\partial t}(x,v,t_0)+\frac{\partial s}{\partial t}(x,v,-t_0)\right)dv,
\end{array}
\]
which implies $\overline{\varphi}_{M,x}^{(2i-1)}(0)=0$ $(1\le 2i-1 \le k)$. 

\smallskip
(2) The same statements for $\overline\varphi_{\nu,x}(t)$ can be proved in the same way. 
\end{proof}

Since the formulae \eqref{beta_closed} and \eqref{beta_body} imply 
\[\begin{array}{l}
\displaystyle I_z(M)
=\int_0^\infty t^{z+m-1}\left(\int_M \overline{\varphi}_{M,x}(t)\, dx\right)dt, \\[4mm]
\displaystyle I_z(\Omega)
=\frac{-1}{(z+2)(z+d)}\int_0^\infty t^{z+d}\left(\int_{\partial\Omega}\overline{\varphi}_{\nu,x}(t)\,dx\right)dt,
\end{array}
\]
the regularization of $I_z(M)$ and $I_z(\Omega)$ can be reduced to that of an integral of the form $I_{w,\phi}=\int_0^\infty t^w\phi(t)\,dt$. 
If $\phi(t)$ is of class $C^k$ then the integrand of the first term of the right hand side of 
\begin{equation}\label{GS}
\begin{array}{rcl}
\displaystyle I_{w,\phi}=\int_0^\infty t^w\,\phi(t)\,dt 
&=& \displaystyle \int_0^1t^w\left[\phi(t)-\phi(0)-\phi'(0)t-\dots -\frac{\phi^{(k-1)}(0)}{(k-1)!}\,t^{k-1}\right]\,dt \\[5mm]
&&\displaystyle +\int_1^\infty t^w\,\phi(t)\,dt 
+\sum_{1\le j\le k}\frac{\phi^{(j-1)}(0)}{(j-1)!\,(z+j)} 
\end{array}
\nonumber
\end{equation}
(\cite{GS} Ch.1, 3.2) can be estimated by $t^{w+k}$, and hence the integral converges for $\R w>-k-1$. 
Therefore $I_{w,\phi}=\int_0^\infty t^w\phi(t)\,dt$ is meromorphic on $\R w>-k-1$ having possible simple poles at $z=-1,\dots, -k$ with the residue at $z=-j$ given by $\phi^{(j-1)}/(j-1)!$ for $j=1,\dots,k$. 
Since $\overline{\varphi}_{M,x}$ and $\overline{\varphi}_{\nu,x}$ are of class $C^{k}$ and $\overline\varphi_{M,x}^{(2i-1)}(0)=\overline\varphi_{\nu,x}^{(2i-1)}(0)=0$, by putting $w=z+m-1$ 
for $M$ or $w=z+d$ 
for $\Omega$, we obtain the following. 

\begin{corollary}\label{cor_regularity}
\begin{enumerate}
\item Suppose $M$ is an $m$ dimensional closed submanifold of $\RR^d$ of class $C^{k+1}$ $(k\ge1)$. The beta function $B_M(z)$ is meromorphic on $\R z> -m-k$ which has possible simple poles at $z=-m-2i$, where $0\le 2i\le k-1$, with
\[
\mbox{\rm Res}(B_M, -m-2i)=\frac1{(2i)!}\int_M \overline{\varphi}_{M,x}^{(2i)}(0) \, dx \hspace{0.5cm} (0\le 2i\le k-1). 
\]
In particular, 
\begin{equation}\label{residue_beta_closed_-m}
\mbox{\rm Res}(B_M, -m)=\sigma_{m-1}\mbox{\rm Vol}\,(M),
\end{equation}
where $\sigma_{j}$ is the volume of the unit $j$-sphere. 
%
%
\item Suppose $\Omega$ is a compact body in $\RR^d$ of class $C^{k+1}$ $(k\ge1)$. The beta function $B_\Omega(z)$ is meromorphic on $\R z>-d-k-1$ which has possible simple poles at $z=-d$ and $z=-d-(2i+1)$, where $1\le 2i+1\le k$, with 
\[
\mbox{\rm Res}(B_\Omega, -d-(2i+1))=\frac{-1}{(d+2i-1)(2i+1)!}\int_{\partial\Omega} \overline{\varphi}_{\nu,x}^{(2i)}(0) \, dx \hspace{0.5cm} (1\le 2i+1\le k). 
\]
In particular, 
\begin{equation}\label{residue_beta_body_-d-1}
\mbox{\rm Res}\,(B_\Omega,-d-1)=-\frac{\sigma_{d-2}}{d-1}\mbox{\rm Vol}\,(\partial \Omega). 
\end{equation}
%
\end{enumerate}
\end{corollary}
\begin{remark}\rm 
\begin{itemize}
\item  The equation \eqref{residue_beta_closed_-m} for smooth closed curves was given in \cite{B}. 
Two formulae of residues and the eqation \eqref{residue_beta_body_-d-1} 
for smooth case were given in \cite{OS2}. 
%
%
\item When $M$ is a closed surface in $\RR^3$, the second residue which appears at $z=-4$ is given by 
\begin{equation}\label{residue_surface_-4}
\mbox{\rm Res}\,(B_M,-4)=\frac\pi 8\int_M(\kappa_1-\kappa_2)^2dx,
\end{equation}
where $\kappa_1$ and $\kappa_2$ are principal curvatures of $M$ (Theorem 4.1 of \cite{FV}; see also Proposition 3.8 of \cite{OS2} for the correction of the coefficient). 
\item The first residue of $B_\Omega(z)$ which appears at $z=-d$ is given by 
\begin{equation}\label{residue_body}
\mbox{\rm Res}\,(B_\Omega,-d)=\sigma_{d-1} \mbox{\rm Vol}\,(\Omega), 
\end{equation}
which can be computed using \eqref{beta_body_Stokes} without using differentiability of $\overline{\varphi}_{\nu,x}(t)$ (\cite{OS2} Lemma 4.5). 
\item The residues of the beta function do not indicate the number of the connected components of $X$ immediately. 
\end{itemize}
\end{remark}

\section{Identification of balls and spheres}
%
Let $B^n(r), S^1(r)$, and $S^2(r)$ be an $n$-ball, circle, and a $2$-sphere of radius $r$ respectively. 

\begin{lemma}\label{lemma_circle}
If $X$ is a disjoint union of closed curves in $\RR^d$, $B_X(-2)\ge0$ with equality if and only if $X$ is a single circle. 
\end{lemma}

\begin{proof}
Brylinski showed that $B_C(-2)=E(C)-4$ for a single curve $C$, where $E(C)$ is the so-called {\sl M\"obius energy} defined in \cite{O1} and studied in \cite{FHW}\footnote{In fact, the energy given in \cite{O1} is equal to $(1/2)B_C(-2)$.}. Freedman, He and Wang showed that $E(C)\ge4$ for any single closed curve $C$ in $\RR^3$ with equality if and only if $C$ is a circle. 
The easiest way to see this would be the ``wasted length'' argument and the cosine formula of $E$ by Doyle and Schramm (reported in \cite{AS}). 

Since the definition of the energy and the proofs of the above statements do not use the condition that the dimension of the ambiet space is equal to $3$, the above argument holds regardless of the codimension.  

Suppose $X$ is a disjoint union of $n$ closed curves; $X=C_1 \cup \dots \cup C_n$. We have
\[
B_X(-2)=\sum_{i=1}^nB_{C_i}(-2)+2\sum_{i<j}\int_{C_i}\int_{C_j}|x-y|^{-2}\,dxdy\ge\sum_{i=1}^nB_{C_i}(-2)\ge0,
\]
where the second equality holds if and only if $C_i$ is a circle for any $i$ and the first equality holds if and only if $n=1$. 
\end{proof}

\begin{lemma}\label{sublemma}
Let $X=S^2_1(r_1)\cup S^2_2(r_2)$ be a disjoint union of two $2$-spheres in $\RR^3$ with radii $r_1$ and $r_2$ such that the diameter of $X$ is not greater than $2$. 
Put 
\[
\Delta_{2-\e}^c=\left\{(x,y)\in\RR^3\times\RR^3\,:\,|x-y|>2-\e\right\}.
\]
Then there are positive constants $\e_1$ and $C$ such that if $0<\e<\e_1$ then 
\begin{equation}\label{ratio_area}
\frac{\mbox{\rm Vol}\left((S^2_1(r_1)\times S^2_2(r_2))\cap\Delta_{2-\e}^c\right)}{\mbox{\rm Vol}\left(S^2_1(r_1)\times S^2_2(r_2)\right)}
< C\e^2.  \nonumber
\end{equation}
\end{lemma}

\begin{proof}
Put $\e_1=\min\{1,r_1,r_2\}$. Since the numerator of the left hand side of \eqref{ratio_area} is an increasing function of the distance between two spheres, we have only to show the inequality when the distance is equal to $1-2r_1-2r_2$. 
Therefore we may assume both $S^2_1(r_1)$ and $S^2_2(r_2)$ are contained in the unit ball with center the origin. 

If $(x,y)\in (S^2_1(r_1)\times S^2_2(r_2))\cap\Delta_{2-\e}^c$ then 
\[
2-\e<|x-y|\le|x|+|y|\le|x|+1, 
\]
which means that $x$ is in the complement of the ball with center the origin and radius $1-\e$, which we denote by $\left(B^3(1-\e)\right)^c$. 
Since 
\[
A\left(S^2_1(r_1)\cap \left(B^3(1-\e)\right)^c\,\right)=2\pi r_1\,\frac{\e(1-\frac\e2)}{1-r_1},
\]
where $A$ means the area. We have 
\[\begin{array}{rcl}
\displaystyle 
\frac{\mbox{Vol}\left((S^2_1(r_1)\times S^2_2(r_2))\cap\Delta_{2-\e}^c\right)}{\mbox{Vol}\left(S^2_1(r_1)\times S^2_2(r_2)\right)}
&\le& \displaystyle \frac{A\left(S^2_1(r_1)\cap \left(B^3(1-\e)\right)^c\,\right)\cdot A\left(S^2_2(r_2)\cap \left(B^3(1-\e)\right)^c\,\right)}{A\left(S^2_1(r_1)\right)\cdot A\left(S^2_2(r_2)\right)}  \\ [4mm]
&=& \displaystyle \frac{\e^2(1-\frac\e2)^2}{4r_1r_2(1-r_1)(1-r_2)},
\end{array}\]
which implies that if we put 
\[
C=\frac{1}{4r_1r_2(1-r_1)(1-r_2)}
\]
then the inequality \eqref{ratio_area} is satisfied. 
\end{proof}

\begin{lemma}\label{lemma_sphere}
Suppose $X$ is a disjoint union of $n$ two dimensional spheres in $\RR^3$ that has the same area and diameter as $S^2(r)$. If $n>1$ then $X$ has a different interpoint distance distribution, and hence a different beta function, from $S^2(r)$. 
\end{lemma}

\begin{proof}
We may assume without loss of generality that $r=1$. 
Assume that $X=S^2_1(r_1)\cup\dots \cup S^2_n(r_n)$ with $n>1$, $r_1\ge \dots \ge r_n$, $r_1^2+\cdots+r_n^2=1$ and that the diameter of $X$ is equal to $2$. 
Put $\e_0=\min\{2-2r_1,r_n,1\}$. 
Then if $0<\e<\e_0$ then $(S^2_i(r_i)\times S^2_i(r_i))\cap\Delta_{2-\e}^c=\emptyset$ for any $i$. Therefore, Lemma \ref{sublemma} implies 
\[
\mbox{Vol}\left((X\times X)\cap\Delta_{2-\e}^c\right)
\le\sum_{i\ne j}\mbox{Vol}\left((S^2_i(r_i)\times S^2_j(r_j))\cap\Delta_{2-\e}^c\right)
\le C_0\,\e^2\sum_{i\ne j}\mbox{Vol}\left(S^2_i(r_i)\times S^2_j(r_j)\right),
\]
where 
\[
C_0=\max_{i\ne j}\frac1{4r_ir_j(1-r_i)(1-r_j)}.
\]
If we take $\e>0$ so that $(C_0+\frac14)\e<1$ then 
\[
\mbox{Vol}\left((X\times X)\cap\Delta_{2-\e}^c\right)\le C_0\,\e^2\,\mbox{Vol}\,(X\times X)
<\left(\e-\frac{\e^2}4\right)\mbox{Vol}\,(S^2\times S^2)
=\mbox{Vol}\left((S^2\times S^2)\cap\Delta_{2-\e}^c\right),
\]
which completes the proof.
\end{proof}

\begin{theorem}\label{main_thm} 
Assume $X$ is a compact submanifold of $\RR^d$ that is either a body $(\dim X=d)$ or a closed submanifold $(\partial X=\emptyset, \, \dim X<d)$. 
Then the following hold up to congruence of $\RR^d$. 
\begin{enumerate}
\item If $X$ is a compact body of class $C^2$ and $B_X(z)=B_{B^d(r)}(z)$ holds for any $z\in\CC$, then $X=B^d(r)$. 

\item If $X$ is of class $C^3$ and $B_X(z)=B_{B^n(r)}(z)$ holds for any $z\in\CC$, then $n=d$ and $X=B^d(r)$. 
\item If $X$ is of class $C^3$ and $B_X(z)=B_{S^1(r)}(z)$ holds for any $z\in\CC$, then $X=S^1(r)$. 
\item If $X$ is of class $C^4$, $d-\dim X\le 1$ and $B_X(z)=B_{S^2(r)}(z)$ for any $z\in\CC$, then $d=3$ and $X=S^2(r)$. 
\end{enumerate}
\end{theorem}

\begin{proof}
We first give proof under the assumption that $X$ is smooth. 

\medskip
(1) By the equations \eqref{residue_body} and \eqref{residue_beta_body_-d-1}, the residues at $z=-d$ and $-d-1$ imply that $\mbox{\rm Vol}(X)=\mbox{\rm Vol}(B^d(r))$ and $\mbox{\rm Vol}(\partial X)=\mbox{\rm Vol}(\partial B^d(r))$. Then the isoperimetric inequality in general dimension (\cite{F}) implies that $X$ is an $d$-ball with radius $r$. 

\smallskip
(2) Suppose $B_X(z)=B_{B^n}(z)$ for any $z\in\CC$. 
The information of the poles implies that $X$ is a compact body in $\RR^n$, and hence $n=d$ by the assumption of the theorem. The rest is same as in (1). 

\medskip
(3) Suppose $B_X(z)=B_{S^{1}}(z)$ for any $z\in\CC$. The information of the poles implies that $X$ is a union of closed curves in $\RR^d$. 
By Lemma \ref{lemma_circle}, $X$ is a single circle. 
By the equation \eqref{residue_beta_closed_-m}, the residue at $z=-1$ implies that $L(X)=2\pi r$, and hence $X=S^1(r)$. 

\medskip
(4) Suppose $B_X(z)=B_{S^{2}(r)}(z)$ for any $z\in\CC$. The information of the poles implies that $X$ is a union of $2$-dimensional closed surfaces, and hence, by the additional assumption of the theorem, $d=3$. Since $\mbox{\rm Res}\,(B_X,-4)=0$, the equation \eqref{residue_surface_-4} shows that $X$ is totally umbilic, which implies that each connected component of $X$ is part of either a sphere or a plane (Meusnier 1785). 
Since $X$ is a closed surface, $X$ is a union of spheres. 
By \eqref{residue_beta_closed_-m}, $\Res(B_X,-2)=\Res(B_{S^2(r)},-2)$ implies that the area of $X$ is same as that of $S^2(r)$. 
Since the diameter of $X$ is given by $\lim_{n\to\infty}(B_X(n))^{1/n}$, $X$ has the same diameter as $S^2(r)$. 
Now the conclusion follows from Lemma \ref{lemma_sphere}. 

\medskip
We next show that the regularity of $X$ specified in each statement of the theorem is enough for the proof. 
Corollary \ref{cor_regularity} implies that if $M$ (or $\Omega$) is of class $C^{k+1}$, we obtain the first $k$ (or respectively, $k+1$) successive residues (including $0$) of $B_M(z)$ (or respectively, $B_\Omega(z)$) starting from $z=-m$ (or respectively, $z=-n$) which gives the first non-zero residue. 
Therefore, the regularity of $X$ guarantees the existence of a necessary number of residues for the proof of each statement. 
\end{proof}

\begin{corollary} 
Under the same assumption as in Theorem \ref{main_thm}, balls, circles, and $2$-spheres can be  identified by the interpoint distance distribution. 
%
%
\end{corollary}

Jun O'Hara

Department of Mathematics and Informatics,Faculty of Science, 
Chiba University

1-33 Yayoi-cho, Inage, Chiba, 263-8522, JAPAN.  

E-mail: ohara@math.s.chiba-u.ac.jp


\begin{thebibliography}{OS} 
\bibitem[AS]{AS} D.~Auckly and L.~Sadun, 
     {\em  A family of M\"obius invariant 2-knot energies}
       Geometric Topology (Proceedings of the 1993 Georgia International Topology Conference)
       AMS/IP Studies in Adv. Math., W. H. Kazez ed. 
       Amer. Math. Soc. and International Press
      addr Cambridge, MA. 
       (1997), 235\,--\,258



\bibitem[B]{B} J.-L. Brylinski, {\em The beta function of a knot.} Internat. J. Math. {\bf 10} (1999), 415\,--\,423. 


\bibitem[Ca]{C}T.~Caelli, {\em On generating spatial configurations with identical interpoint distance distributions}, in: Combinatorial Mathematics, VII (Proc. Seventh Australian Conf., Univ. Newcastle, Newcastle, 1979), in: Lecture Notes in Math., vol. 829, Springer, Berlin (1980),  69\,--\,75.


\bibitem[D]{D}P.~Davy, {\em Inequalitiesf or moments of secant length}, Z. Wahrscheinlichkeitsth 68 (1984), 243\,--\,246.

\bibitem[Fed]{F}H.~Federer, {\em Geometric measure theory.} Springer (1969). 

\bibitem[Fen]{fenchel}W.~Fenchel, {\em \"Uber Kr\"ummung und Windung geschlossener Raumkurven.} Math. Ann. 101 (1929), 238\,--\,252. 

\bibitem[FHW]{FHW}M. H.~Freedman, Z-X.~He and Z.~Wang, {\em M\"obius 
energy of knots and unknots.} Ann. of Math. \textbf{139} (1994), 1\,--\,50.

\bibitem[FV]{FV}E. J.~Fuller and M.K.~Vemuri. {\em The Brylinski Beta Function of a Surface.}  Geometriae Dedicata 179 (2015),  153\,--\,160, doi:10.1007/s10711-015-0071-y. 

\bibitem[G]{G}R.J.Gardner, {\em Geometric Tomography}, second edition, Cambridge University Press, New York, 2006. 

\bibitem[GS]{GS}I.M. Gel'fand and G.E. Shilov, {\em Generalized Functions. Volume I: Properties and Operations}, Academic Press, New York and London, 1967.



\bibitem[MC]{MC}C.~L.~Mallows and J.~M.~C.~Clark, {\em Linear-Intercept Distributions Do Not Characterize Plane Sets.} J. Appl. Prob. 7 (1970), 240\,--\,244.

\bibitem[M]{M}B.~Mat\'ern, {\em Spatial variation}, Springer-Verlag, Berlin (1985). 

\bibitem[O]{O1}J.~O'Hara, {\em Energy of a knot.} Topology  {\bf 30}  (1991), 241\,--\,247. 

\bibitem[OS]{OS2}J.~O'Hara and G.~Solanes, {\em Regularized Riesz energies of submanifolds.} preprint, arXiv:1512.07935. 


\bibitem[San1]{San}L.A.~Santal\'o, {\em On the measure of line segments entirely contained in a convex body}, In Aspects of Mathematicsa nd Its Applications (North-Holland. Math. Library 34), North-Holland, Amsterdam (1986), 677\,--\,687.

\bibitem[San2]{San2}L.A.~Santal\'o, {\em Integral Geometry and Geometric Probability}, Addison- Wesley Publishing Company, 1976.

\bibitem[Sch]{Sch}R.~Schneider, {\em Inequalitiesf or random flats meeting a convex body}, J. Appl. Prob. 22 (1985), 710\,--\,716.


\bibitem[W]{W}P.~Waksman, {\em Polygons and a conjecture of Blaschke's}, Adv. Appl. Prob. 17 (1985), 774\,--\,793.


\end{thebibliography}
\end{document}